\documentclass[12pt,letterpaper]{article}

\usepackage{amssymb,amsfonts,amscd,amsthm}
\usepackage[all,arc]{xy}
\usepackage{enumerate, bbm}
\usepackage{mathrsfs}
\usepackage{tikz}
\usepackage[margin=1.2in]{geometry} 
\usepackage{mathtools}
\usepackage{hyperref}
\usepackage{color}
\usepackage{graphicx}
\usepackage{enumitem} 
\graphicspath{ {TexImages/} }

\newtheorem{thm}{Theorem}[section]
\newtheorem{cor}[thm]{Corollary}

\newtheorem{lem}[thm]{Lemma}
\newtheorem{claim}[thm]{Claim}

\newtheorem{definition}[thm]{Definition}

\hypersetup{
	colorlinks,
	citecolor=blue,
	filecolor=blue,
	linkcolor=blue,
	urlcolor=blue,
	linktocpage
}

\setenumerate[1]{label=\thesection.\arabic*.} 
\setenumerate[2]{label*=\arabic*.} 
\setlist[enumerate]{itemsep=2ex, topsep=2ex} 
\setlist[itemize]{itemsep=2ex, topsep=2ex}

\newcommand{\E}{\mathbb{E}}

\newcommand{\al}{\alpha}

\newcommand{\ep}{\varepsilon}

\newcommand{\del}{\delta}
\newcommand{\Del}{\Delta}

\renewcommand{\l}{\left}
\renewcommand{\r}{\right}

\newcommand{\half}{\frac{1}{2}}

\newcommand{\sm}{\setminus}
\newcommand{\sub}{\subseteq}

\newcommand{\rec}[1]{\frac{1}{#1}}
\newcommand{\f}[2]{\frac{#1}{#2}}

\newcommand{\ceil}[1]{\l\lceil #1\r\rceil}

\newcommand{\mr}[1]{\mathrm{#1}}

\newcommand{\ex}{\mr{ex}}
\newcommand{\Tr}[2]{\mr{Tr}_{#2}(#1)}
\newcommand{\Ber}[2]{\mr{B}_{#2}(#1)}

\title{Forbidding $K_{2,t}$ traces in triple systems}
\author{
{{Ruth Luo}}\thanks{University of Califonia, San Diego, La Jolla, CA 92093, USA. E-mail: {\tt ruluo@ucsd.edu}.
Research is supported in part by NSF grants  DMS-1600592 and DMS-1902808.
}
\and{{Sam Spiro}}\thanks{University of Califonia, San Diego, La Jolla, CA 92093, USA. E-mail: {\tt sspiro@ucsd.edu}.
Research is supported in part by NSF grants DGE-1650112.}}
\date{\today}
\begin{document}
	\maketitle
	\begin{abstract}
	    Let $H$ and $F$ be hypergraphs. We say $H$ {\em contains $F$ as a trace} if there exists some set $S \subseteq V(H)$ such that $H|_S:=\{E\cap S: E \in E(H)\}$ contains a subhypergraph isomorphic to $F$. In this paper we give an upper bound on the number of edges in a $3$-uniform hypergraph that does not contain $K_{2,t}$ as a trace when $t$ is large. In particular, we show that\[ \lim_{t\to \infty}\lim_{n\to \infty} \frac{\mathrm{ex}(n, \mathrm{Tr}_3(K_{2,t}))}{t^{3/2}n^{3/2}} = \frac{1}{6}.\]
	    Moreover, we show $\frac{1}{2} n^{3/2} + o(n^{3/2}) \leq \mathrm{ex}(n, \mathrm{Tr}_3(C_4)) \leq \frac{5}{6} n^{3/2} + o(n^{3/2})$.  
	    	\end{abstract}
\section{Introduction}

A {\em hypergraph} $H$ is a family of subsets of some fixed ground set. The subsets are called the {\em edges} of $H$ and the ground set is called the {\em vertex set} of $H$. We denote these sets by $E(H)$ and $V(H)$ respectively. If each edge of $H$ contains exactly $r$ elements, then we say that $H$ is {\em $r$-uniform}.  

A cornerstone of extremal combinatorics is the {\em Tur\'an problem}. Broadly speaking, the Tur\'an problem asks to determine the maximum number of edges in a hypergraph which contains no subhypergraphs isomorphic to a member of some given forbidden family. In this paper, we study uniform hypergraphs with forbidden traces.

\begin{definition}Let $F$ and $T$ be uniform hypergraphs (possibly with different uniformities) with $V(F) \subseteq V(T)$. We say that $T$ is a {\bf trace of $F$ on $V(F)$}, or simply an $F$-trace, if there exists a bijection $\phi: E(F) \to E(T)$ such that for every edge $e \in E(F)$, $ \phi(e) \cap V(F) = e$. We say a hypergraph $H$ {\bf contains $F$ as a trace} if it contains a subhypergraph isomorphic to a trace of $F$.
\end{definition}
Equivalently, $H$ containing $F$ as a trace means that there exists some set $S$ of vertices (corresponding to $V(F)$) such that $H|_S:= \{E \cap S: E \in E(H)\}$ has a subhypergraph isormorphic to $F$.  We note that in different contexts, traces are also called configurations \cite{SS} and induced Berge $F$'s~\cite{FL}.

For $r \geq 2$, let $\Tr{F}{r}$ denote the set of all $r$-uniform hypergraphs that are traces of $F$ up to isomorphism. If $\mathcal F$ is a family of $r$-uniform hypergraphs, then the function $\ex(n, \mathcal F)$ denotes the maximum number of edges in an $n$-vertex, $r$-uniform hypergraph with no subhypergraph isomorphic to a member of $\mathcal F$. In particular, $\ex(n, \Tr{F}{r})$ is the maximum size of a hypergraph that does not contain $F$ as a trace.

Forbidding traces in hypergraphs is closely related to the well known {\em Berge Tur\'an problem}.

\begin{definition}Given hypergraphs $F$ and $T$, we say $T$ is a {\em Berge $F$} if there exists a bijection $\phi: E(F) \to E(T)$ such that for every edge $e \in E(F)$, $e \subseteq \phi(e)$.
\end{definition}

\begin{figure}
    \centering
    \includegraphics[scale=.15]{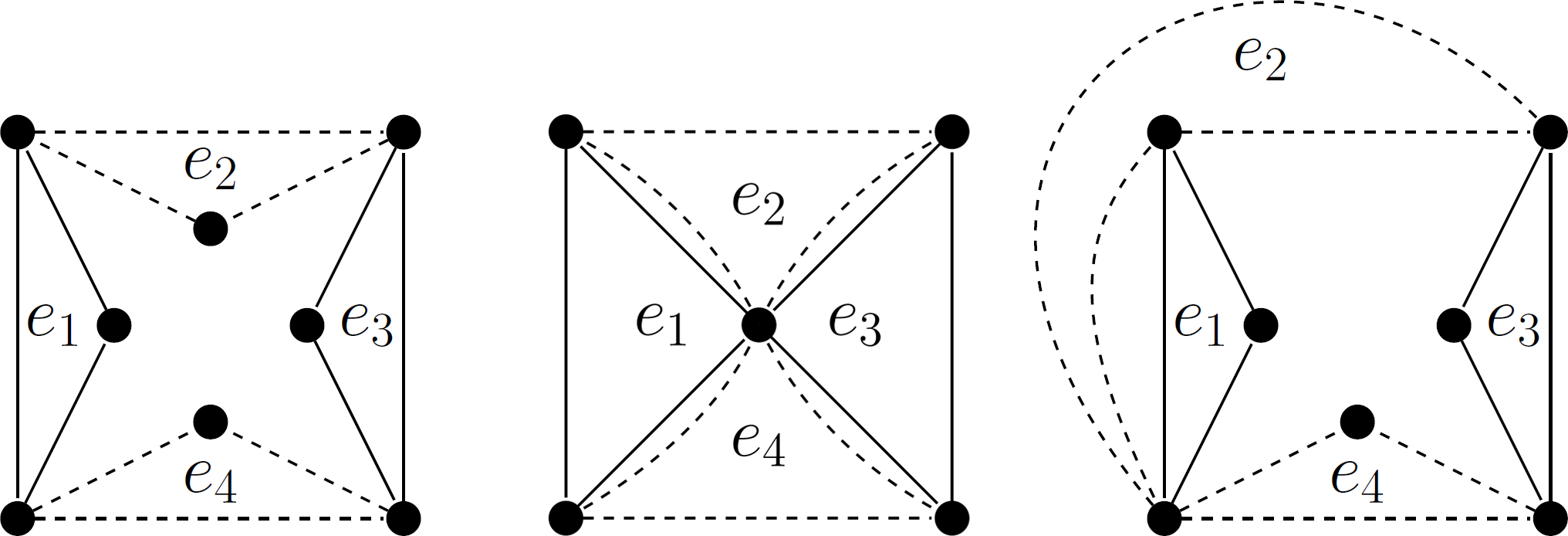} 
    \caption{Two examples of $C_4$ traces and a Berge $C_4$ that is not a trace.}
\end{figure}

Let $\Ber{F}{r}$ denote the set of all $r$-uniform hypergraphs that are Berge $F$'s up to isomorphism. Observe that $\Tr{F}{r} \subseteq \Ber{F}{r}.$ Consequently, 
\begin{equation}\label{bergebound}\ex(n, \Ber{F}{r}) \leq ex(n, \Tr{F}{r}).
\end{equation}

In this paper, we focus only on the case where $F$ is a graph, particularly $F = K_{2,t}$.
\subsection{Known extremal results for degenerate graphs}
Generalizing a result of Mantel~\cite{M}, Tur\'an~\cite{turan} determined $ex(n, K_t)$, the maximum number of edges in an $n$-vertex graph without a copy of $K_t$, for all $t$. Later results by Erd\H{o}s, Stone, and Simonovits~\cite{simonovits, stone} established the asymptotic value of $\ex(n, F)$ for any graph $F$ which is nonbipartite.

Determining $\ex(n,F)$ when $F$ is bipartite is a main area of research in extremal graph theory.  The case $F = K_{2,t}$ is of particular interest to this paper, especially for $F = K_{2,2} = C_4$.  It is known that $
\ex(n, C_4) = \frac{1}{2}n^{3/2}+o(n^{3/2})$, due to
K\"ovari, S\'os, Tur\'an~\cite{KST}; Brown~\cite{brown}; and Erd\H{o}s, R\'enyi, S\'os~\cite{ERS}. These results were further strengthened by F\"uredi~\cite{furedi} to $K_{2,t}$.
\begin{thm}[F\"uredi~\cite{furedi}]\label{thm:graphK2t}For $t \geq 2$, \[\ex(n, K_{2,t}) = \frac{\sqrt{t-1}}{2} n^{3/2} + O(n^{4/3}).\]
\end{thm}



The Tur\'an problems for even cycles is among the most famous open problems in graph theory. Bondy and Simonovits~\cite{BS} proved that $\ex(n,C_{2k}) = O(n^{1+1/k})$, but lower bounds with matching orders of growth exist only for $k \in \{2,3,5\}$.  For hypergraphs, Gy\H{o}ri and Lemons~\cite{GL} proved that, as in the graph case, for $r \geq 3$, $k \geq 2$, $\ex(n, \Ber{C_{2k}}{r}) = O(n^{1+1/k})$. They also proved $\ex(n, \Ber{C_{2k+1}}{r}) = O(n^{1+1/k})$, which is rather surprising given that $\ex(n, C_{2k+1})=\Theta(n^2)$. As in the graph case, finding constructions for lower bounds is difficult. For example, there are no known constructions of $r$-uniform, Berge $C_4$-free hypergraphs with $\Theta(n^{3/2})$ edges when $r \geq 6$. See~\cite{GP, GL2, EGM, ergem, BG} for more related results.

In~\cite{GMV}, Gerbner, Methuku, and Vizer proved an upper bound for the Tur\'an number of Berge $K_{2,t}$ in $r$-uniform hypergraphs. They obtained asymptotically sharp bounds for $3$-uniform hypergraphs when $t \geq 7$. This was later extended by Gerbner, Methuku, and Palmer~\cite{GMP} for $t \geq 4$. 

\begin{thm}[Gerbner--Methuku--Palmer~\cite{GMP}] \label{bergek2t}
For $t \geq 4$, 
\[\ex(n, \Ber{K_{2,t}}{3})= \frac{1}{6}(t-1)^{3/2}n^{3/2}+o(n^{3/2}).\]
\end{thm}
Focusing on the $r=3,t=2$ case, early upper bounds for $\ex(n,\Ber{C_4}{3})$ were implied by results of Alon and Shikhelman~\cite{AS} and F\"uredi and {\"O}zkahya~\cite{FO}. Currently the best known bound is $\ex(n, \Ber{C_4}{3}) \leq \frac{1}{\sqrt{10}}n^{3/2} + o(n^{3/2})$ due to Ergemlidze, Gy\H{o}ri, Methuku, Tompkins, and Salia~\cite{EGMTS}.

\subsection{Known results for forbidden traces}
Some earlier results for forbidding traces of graphs in hypergraphs where due to Mubayi and Zhao~\cite{MZ} who determined the asymptotic value of $\ex(n, \Tr{K_s}{r})$ for all $r$ when $s \in \{3,4\}$. They also conjectured that for $ s \geq 5$, $\ex(n, \Tr{K_s}{r}) \sim \left(\frac{n}{s-1}\right)^{s-1}$.

Sali and Spiro~\cite{SS} determined the order of magnitude of $\ex(n, \Tr{K_{s,t}}{r})$ when $t\geq (s-1)!+1,\ s \geq 2r-4$. Later, F\"uredi and Luo~\cite{FL} generalized their proof to deduce the order of magnitude of $ex(n, \Tr{F}{r})$ for all graphs $F$ in terms of their generalized Tur\'an numbers. In particular, they showed
\[\ex(n, \Tr{F}{r}) = \Theta( \max_{2\leq s\leq r} \;\; \ex(n, K_s, F)),\]
where $\ex(n, K_s, F)$ denotes another extremal function, the maximum number of copies of $K_s$ in an $F$-free graph on $n$ vertices.

When $F$ is non-bipartite, this implies $\ex(n, \Tr{F}{r}) = \Omega(n^2)$ for all $r$. This contrasts with the problem of forbidding Berge copies of $F$: Grosz, Methuku, and Tompkins~\cite{GMT} proved that for all $F$ there exists an $r_0$ such that for all $r \geq r_0$, $\ex(n, \Ber{F}{r}) = o(n^2)$. In particular, for large $r$, $\ex(n, \Ber{F}{r}) = o(\ex(n, F)).$

In the case where $F$ is outerplanar, bounds were obtained in terms of the Tur\'an number of $F$. 
\begin{thm}[F\"uredi--Luo~\cite{FL}]\label{thm:old}
    If $F$ is a $t$-vertex outerplaner graph, then \[\ex(n-r+2, F)\leq \ex(n,\Tr{F}{r})\le  \half r^r (t-2)^{r-2}\ex(n,F).\]
\end{thm}

For $F = C_4$ this gives the bounds \begin{equation}\label{eq:oldC4}\half n^{3/2}+o(n^{3/2}) \le \ex(n,\Tr{C_4}{3})\le 27n^{3/2}+o(n^{3/2}).\end{equation}


\section{New Results}

Our main result is an upper bound for $\ex(n, \Tr{K_{2,t}}{3})$ which is effective for large $t$.  Here and throughout $\log$ denotes the natural logarithm.

\begin{thm}\label{thm:main}
	For $t\ge 14$,
	\[\ex(n,\Tr{K_{2,t}}{3})\le \rec{6}\Big(t^{3/2}+55 t\sqrt{\log t}\Big)n^{3/2}+o(n^{3/2}).\]
\end{thm}

We note that the constant 55 can be improved by a more careful analysis.  On the other hand, for $t\ge 4$ we have $\ex(n, \Tr{K_{2,t}}{r}) \geq \frac{1}{6}(t-1)^{3/2}n^{3/2}+o(n^{3/2})$ by \eqref{bergebound} and Theorem~\ref{bergek2t}.  This together with Theorem~\ref{thm:main} gives the following.

\begin{cor}\label{asymp}
\[\lim_{t\to \infty}\lim_{n\to \infty} \frac{\ex(n, \Tr{K_{2,t}}{3})}{t^{3/2}n^{3/2}} = \frac{1}{6}.\]
\end{cor}

Separately analysing the case for $K_{2,2} = C_4$,  we obtain tighter bounds which significantly improves \eqref{eq:oldC4}.
\begin{thm}\label{c4}
\[\frac{1}{2} n^{3/2} + o(n^{3/2}) \leq \ex(n,\Tr{C_4}{3})\le \f{5}{6}n^{3/2}+o(n^{3/2}).\]
\end{thm}



\section{Main Lemmas and the Proof of Theorem~\ref{thm:main}}
Given a hypergraph $H$, we define $d_H(x,y)$ to be the number of edges of $H$ containing $\{x,y\}$, and we call this number the co-degree of $\{x,y\}$.  We will often identify hypergraphs by their set of edges and write e.g. $H\sm A$ to denote the hypergraph $H$ after deleting some set of edges $A$ from $E(H)$.

For a hypergraph $H$ and $\delta \in \mathbb R^+$, define \[H_\delta^+:=\{e \in H: d_H(x,y) > \delta \text{ for all } \{x,y\} \subseteq e\}; \qquad H_\delta^- = H \setminus H_\delta^+.\]
That is, every edge in $H_\delta^-$ contains a pair with co-degree at most $\delta$.

Let $H$ be some 3-uniform hypergraph and fix $\delta \ge 2$. We partition the edges of $H$ into sets with small, medium, and large co-degrees in the following manner:
\begin{itemize}
\item[---]$A = H_1^-$, i.e., $A$ is the set of edges containing at least one pair with co-degree 1;
\item[---]$B_\delta = H_\delta^- \setminus A$, i.e., $B_\delta$ is the set of edges in which every pair has co-degree at least 2, but at least one pair of co-degree at most $\delta$ in $H$; and 
\item[---] $C_\delta= H \setminus (A \cup B_\delta) = H_\delta^+$, i.e., $C_\delta$ is the set of edges in which every pair is contained in at least $\delta$ other edges of $H$.
\end{itemize}
The bulk of the work in showing Theorem~\ref{thm:main} will be in proving the following technical lemmas.  For ease of notation, for $\del\ge 2$ we define  \begin{equation*}\label{eq:ep}\ep_\del = \frac{1+\log(\del+1)}{\del+1},\end{equation*}
and when $\del$ is clear from context we simply write $\ep$.

\begin{lem}\label{lem:cod}
Let $t\geq 2$ and let $H$ be a $\Tr{K_{2,t}}{3}$-free 3-uniform hypergraph on $[n]$.  For any pair $\{x,y\}$, we have $d_{H \setminus A}(x,y) \leq 3t-3$. Moreover, if $t=2$ then $d_{H\setminus A}(x,y) \leq 2$. 
\end{lem}

\begin{lem}\label{lem:medium}Let $t \geq 4$ and let $H$ be a $\Tr{K_{2,t}}{3}$-free 3-uniform hypergraph on $[n]$. For $\delta, k \ge 2$, if $B_\delta$ has maximum co-degree $k$, then
\[e(B_\delta)\le \delta\cdot \half (k+3t-3)^{1/2}n^{3/2}+o(n^{3/2}).\]
\end{lem}

\begin{lem}\label{lem:largecod}Let $t \geq 2$ and let $H$ be a $\Tr{K_{2,t}}{3}$-free 3-uniform hypergraph on $[n]$. If $\del\ge 14$, then for any pair $\{x,y\}$ we have \[d_{C_\delta}(x,y) \leq \left(1 + 4\ep\right)t-1.\] 
\end{lem}
\begin{lem}\label{lem:large}Let $t \geq 2$ and let $H$ be a $\Tr{K_{2,t}}{3}$-free 3-uniform hypergraph on $[n]$. For any $\del\ge 14$ and $k\ge (1+4\ep)t$, if $C_\delta$ has maximum co-degree at most $k$, then
\[e(C_\delta)\le \frac{1}{6} k^{3/2} n^{3/2} + +o(n^{3/2}).\]
\end{lem}
Assuming these lemmas, we can prove the following technical theorem.
\begin{thm}\label{k2t}Fix $t$ and let $g(t)$ be any function such that $14 \leq t/g(t) \leq t$. Then
	\[\ex(n, \Tr{K_{2,t}}{3}) \leq \frac{1}{2}\sqrt{t-1}n^{3/2}+\frac{\sqrt{6}}{2} \cdot \frac{t^{3/2}}{g(t)} n^{3/2} +\frac{1}{6} (t + 5g(t) \log(t))^{3/2} n^{3/2} + o(n^{3/2}).\]
\end{thm}

Let us first show that this implies our main result.
\begin{proof}[Proof of Theorem~\ref{thm:main}, assuming Theorem~\ref{k2t}]
	Take $g(t)=\rec{7}\sqrt{t\log(t)}$, and note that $t \geq t/g(t) = 7\sqrt{t/\log(t)} \ge 14$ when $t \geq 14$, so we can apply the bound of Theorem~\ref{k2t}.  
	Because $\sqrt{t-1}\le t\log(t)^{-1/2}$ for $t\ge 14$, we have
	\[\half \sqrt{t-1}+\f{\sqrt{6}t^{3/2}}{2g(t)}\le \l(\half+\f{7\sqrt{6}}{2}\r)t\log(t)^{-1/2}\le \rec{6}\cdot 52 t\log(t)^{-1/2}.\]
	We also have \begin{align*}\frac{1}{6} (t + 5g(t) \log(t))^{3/2}&=\rec{6}t^{3/2}\l(1+\f{5}{7}(t\log(t))^{-1/2}\r)^{3/2}\le \rec{6}t^{3/2}(1+(t\log(t))^{-1/2})^{2}\\&\le \rec{6}t^{3/2}(1+3(t\log(t))^{-1/2})=\rec{6}(t^{3/2}+3t\log(t)^{-1/2}).\end{align*}
	
	Combining this with the inequality above gives the desired result.
\end{proof}
It remains to prove Theorem~\ref{k2t}.
\begin{proof}[Proof of Theorem~\ref{k2t}, assuming Lemmas~\ref{lem:cod} --~\ref{lem:large}]
Let $G$ be a graph on $[n]$ whose edge set is obtained by selecting from each $e \in A$ a pair of vertices with co-degree 1. Suppose there exists a subgraph $K\sub G$ which is a copy of $K_{2,t}$. For every edge $xy$ in $K$, there exists some vertex $z$ such that $\{x,y,z\}\in A$.  Note that we can not have $z\in V(K)$, as if say $xz\in E(K)$, then this implies there exists some other edge $\{x,z,w\}\in A$ and hence $d_H(x,z)>1$, a contradiction to how $A$ was defined. Therefore the edges of $A$ corresponding to $K$ in $G$ intersect $V(K)$ in exactly the edges of $K$.  This forms a $K_{2,t}$ trace in $H$, a contradiction.  We conclude by Theorem~\ref{thm:graphK2t} that
\begin{equation}\label{E1}
e(A) = e(G) \leq ex(n, K_{2,t}) \leq  \frac{\sqrt{t-1}}{2} n^{3/2} + o(n^{3/2}).
\end{equation}

Now set $\delta =t/g(t)\ge 14$. By Lemma~\ref{lem:cod} we conclude that the hypergraph induced by $B_\del\sub H\sm A$ has maximum co-degree at most $k=3t-3$.  Thus by Lemma~\ref{lem:medium} we obtain
\begin{equation}\label{E2}
e(B_\delta) \leq \frac{\delta}{2} (6t)^{1/2}n^{3/2} + o(n^{3/2}).
\end{equation}

From Lemmas~\ref{lem:largecod} and~\ref{lem:large} with $k = (1 + 4\ep)t \leq (1 + 5 \log(\delta)\delta^{-1})t$, we get
\begin{equation}\label{E3}
e(C_\delta) \leq \frac{1}{6} ((1 + 5 \log(\delta)\delta^{-1})t)^{3/2}n^{3/2} + o(n^{3/2}).
\end{equation}

For $\delta = t/g(t)$,
\begin{eqnarray*}e(C_\delta) &\leq & \frac{1}{6} (t + 5\log(t/g(t))g(t))^{3/2} n^{3/2} +o(n^{3/2})\\
&\leq &  \frac{1}{6} (t + 5 g(t) \log t)^{3/2} n^{3/2} +o(n^{3/2}),
\end{eqnarray*}
where the last inequality uses the fact that $t/g(t) \leq t$. Combining this with~\eqref{E1},~\eqref{E2}, and \eqref{E3} gives the desired result.
\end{proof}

The rest of the paper is organized as follows.  In Section~\ref{sec:dom} we introduce the notion of dominated sets and prove Lemmas~\ref{lem:cod}--\ref{lem:largecod}.  We prove Lemma~\ref{lem:large} in Section~\ref{sec:large}.  Finally, focusing on the $t=2$ case, we prove Theorem~\ref{c4} in Sections~\ref{sectionc4}.

We gather some standard notation we use throughout the paper.  For a graph $G$ we let $\Del(G)$ denote its maximum degree, and we define $\al(G)$ to be the size of a maximum independent set in $G$.  $N_G(v)$ denotes the neighborhood set of $v$ in $G$.  We will write edges either as $\{x,y\}$ or $xy$ depending on the context, and similarly for hyperedges we write either $\{x,y,z\}$ or $xyz$.

\section{Dominated Sets and Co-Degrees}\label{sec:dom}
In the literature, a set $A$ of vertices in a graph is a {\em dominating set} if every vertex is either in $A$ or has a neighbor in $A$. Here we introduce the notion of dominated sets in graphs with loops. Suppose $G$ is a graph, possibly with loops. We say $D \subseteq V(G)$ is a {\em dominated set} if for every $v \in D$, either $v$ has a loop edge or $v$ has a neighbor outside of $D$.  We define the degree of a vertex $v$ in such a graph to be the number of edges incident to $v$ (counting loops with multiplicity). 

Given a hypergraph $H$, distinct vertices $x,y\in V(H)$, and $S\sub V(H)\sm \{x,y\}$, we define the graph $L_{x}=L_{x}(H,S,y)$ on $S$ by adding an edge $uv$ with $u,v\in S$ if $\{u,v,x\}\in E(H)$, and we add a loop to $u$ for each edge of the form $\{u,v,x\}\in E(H)$ with $v\notin S\sm \{y\}$. The key observation is the following.

\begin{lem}\label{lem:simDom}
	Let $H$ be a 3-uniform hypergraph $x,y\in V(H)$, and $S\sub V(H)\sm \{x,y\}$.  If there exists a set $D\sub S$ of size $t$ such that $D$ is dominated in both $L_x=L_x(H,S,y)$ and $L_y=L_y(H, S, x)$, then $H$ contains a $K_{2,t}$ trace.
\end{lem}
\begin{proof}
	By assumption of $D$ being dominated in $L_x$, for all $u\in D$ there exists a vertex $u_x$ with $\{x,u,u_x\}\in E(H)$ such that either $u_x\notin S$ (if $u$ has a loop) or $u_x \in S$ but $u_x\notin D$.  Similarly one can find edges of the form $\{y,u,u_y\}$ which intersect $\{x,y\}\cup D$ in exactly two vertices.  This gives a $K_{2,t}$ trace in $H$ with vertex set $\{x,y\} \cup D$ and edge set $\{\{x,u,u_x\}:u \in D\} \cup \{\{y,u,u_y\}:u \in D\}.$
\end{proof}
The other important observation we make is
\begin{equation}\label{eq:deg}
d_{L_x}(u)\ge d_{H}(x,u)-1.
\end{equation}
Indeed, every edge $\{x,u,v\}\in E(H)$ contributes to an edge involving $u$ in $L_x$ (possibly as a loop) unless $v=y$.

Thus our goal is to find large sets that are simultaneously dominated in two graphs.  The most general lemma we have in this direction is the following, which is an easy adaptation of a standard proof for finding a small dominating set (see for example~\cite{alonspencer}).  Recall that we define $\ep=\ep_\delta=\f{1+\log(\del+1)}{\del+1}$.

\begin{lem}\label{lem:minDel}
	Let $G$ be an $n$-vertex graph with loops with minimum degree at least $\del\ge 2$.  Then $G$ has a dominated set of size at least $(1-\ep)n$.
\end{lem}
\begin{proof}
	Let $D\sub V(G)$ be a random set obtained by picking each vertex of $G$ independently with probability $p$. Let $T \subseteq D$ be the set of vertices of $D$ which do not have loops and do not have neighbors outside of $D$.  Observe that $D\sm T$ is a dominated set.
	
	Any given $v\in V(G)$ is in $T$ with probability 0 if it has a loop and otherwise with probability at most $p^{\del+1}$, as all its neighbors and itself must be selected. Thus by linearity of expectation we have \[\E[|D\sm T|]\ge (p-p^{\del+1})n\ge pn-e^{-(1-p)(\del+1)}n.\]  Taking $p=1-\log(\del+1)/(\del+1)$  gives a set of size at least $n-\f{1+\log(\del+1)}{\del+1}n$.
\end{proof}
This quickly gives an upper bound for the co-degrees of $C_\del$.
\begin{proof}[Proof of Lemma~\ref{lem:largecod}]
	Assume we have some pair of vertices $\{x,y\}$ and a set $S$ of size at least $(1 + 4\ep)t$ such that $\{x,y,u\}\in C_\del$ for all $u\in S$.  By definition of $C_\del$, this implies that $d_H(x,u),d_H(y,u)>\del$ for all $u\in S$.  By \eqref{eq:deg} and Lemma~\ref{lem:minDel}, we can find sets $D_x,D_y\sub S$ which are dominated in $L_x,L_y$ of size at least $(1-\ep)(1+4\ep)t$, and in particular $D:=D_x\cap D_y$ will be dominated in both and have size at least $(1-2\ep)(1+4\ep)t\ge t$,	where we use that $\ep \leq 1/4$ whenever $\delta \geq 14$.  Then $H$ contains a $K_{2,t}$ trace by Lemma~\ref{lem:simDom}, a contradiction.
\end{proof}

We next want to prove a bound when $L_x,L_y$ are only known to have minimum degree at least 1.  We first need the following simple result, where we recall that $G'\sub G$ is called a spanning subgraph if $V(G')=V(G)$.

\begin{lem}\label{lem:star}
	Let $G$ be a graph with loops and minimum degree at least 1.  Then there is a spanning subgraph $G'\sub G$ such that every connected component is either a vertex with a loop or a star with at least 2 vertices.
\end{lem}
\begin{proof}
	
	
	We greedily build our subgraph. Suppose at step $i$, we have a subgraph with components $S_1, \ldots, S_{i-1}$ such that each component is either a star or a vertex with a loop. Let $V_{i-1}$ be the set of vertices covered by $S_1, \ldots, S_{i-1}$, and suppose there exists $v \in V(G) \setminus V_{i-1}$. If $v$ has a loop, we set $S_i = \{v\}$. Otherwise, let $S_i$ be the star with center $v$ and leaf vertices $N(v) \setminus V_{i-1}$. Then $S_i$ has at least 1 leaf unless $N_G(v) \setminus V_{i-1}$ is empty. In this case, for any $u \in N_G(v)$ we have $u \in S_j$ for some $j \leq i-1$. If $S_j = \{u\}$, that is, $u$ has a loop, remove $S_j$ and let $S_i = vu$. So suppose $S_j$ is a star. Note that $u$ is not the center, otherwise $v$ would also be in $S_j$. If $S_j$ has at least two leaves, then we replace it with the star $S_j \setminus \{u\}$ and let $S_i = vu$. Otherwise $S_j$ is a single edge, say $wu$. Then we remove $S_j$ and let $S_i$ be the star with edges $uv, uw$. 
\end{proof}
With this we can prove the following.
\begin{lem}\label{lem:min1}
	Let $G_x,G_y$ be graphs on $S$ with minimum degree at least 1.  Then there exists a set $D$ which is dominated in both $G_x$ and $G_y$ of size at least $|S|/3$.  Moreover, if $|S|=3$, then one can find such a set with $|D|=2$ unless $G_x\cup G_y$ is a $K_3$ (possibly with loops).
\end{lem}
\begin{proof}
	Let $G'_x\sub G_x,\ G'_y\sub G_y$ be the subgraphs guaranteed by Lemma~\ref{lem:star}.  Let $C_x$ consist of the centers of stars of order at least 2 in $G_x'$, where a center of a star of order 2 is chosen arbitrarily.  Similarly define $C_y$ and set $D=S\sm (C_x\cup C_y)$.  Note that by assumption on $G_x'$, every $u\in D\sub S\sm C_x$ either has a loop or is adjacent to something in $C_x$.  The same holds for $G_y'$, so $D$ is dominated in both graphs and hence also in $G_x,G_y$. 
	
	It remains to bound the size of $D$. Observe that every $u\in D$ is adjacent to at most one vertex in each of $G_x',G_y'$ (namely the center of the star it's in).  Thus for each vertex added to $D$ we omitted at most two vertices from $D$, giving the first bound.  If, say, $S=\{u,v,w\}$ and $uv\notin G_x\cup G_y$, then by the minimum degree conditions, each of $u,v$ must be adjacent to either $w$ or have a loop in $G_x,G_y$.  Thus $D=\{u,v\}$ is a dominated set.
\end{proof}

We now prove Lemmas~\ref{lem:cod} and~\ref{lem:medium}.
\begin{proof}[Proof of Lemma~\ref{lem:cod}]
	Assume there exists a set $S$ of size at least $3t-2$ and a pair $\{x,y\}$ such that $\{x,y,u\}\in E(H)\sm A$ for all $u\in S$.  By definition of $A$, this implies that $d_H(x,u),d_H(y,u)\ge 2$ for all $u\in S$.  Thus $L_x=L_x(H,S,y)$ and $L_y=L_y(H,S,x)$ have minimum degree at least 1 by \eqref{eq:deg}, so by Lemma~\ref{lem:min1} we can find a set $D\sub S$ which is simultaneously dominated in $L_x,L_y$ of size at least $\ceil{|S|/3}\ge t$. Then $H$ contains a $K_{2,t}$ trace by Lemma~\ref{lem:simDom}, a contradiction.
	
	For $t=2$, if there exists such an $S=\{u,v,w\}$, then by Lemma~\ref{lem:min1} we can assume $L_x\cup L_y$ is a $K_3$, and without loss of generality we can assume $uv,uw\in L_x$.  By definition this implies that $\{x,u,v\},\{x,u,w\}\in E(H)$.  By definition of $S$ there exist edges $\{x,y,v\},\{x,y,w\}\in E(H)$.  These four edges form a $C_4$ trace on $\{y,u,v,w\}$, which is a contradiction.
\end{proof}

\begin{proof}[Proof of Lemma~\ref{lem:medium}]
Recall that $B_\delta$ is the set of edges of $H\setminus A$ that contain a pair with co-degree at most $\delta$. Let $G$ be the graph on $[n]$ whose edge set is obtained from $B_\delta$ by taking from each $e \in B_\delta$ a pair of vertices with co-degree at most $\delta$ in $H\setminus A$. Observe that $e(B_\delta)\le \delta \cdot e(G)$ as each edge in $G$ is mapped to by at most $\delta$ edges of $H$.  We claim that $G$ is $K_{2,r}$-free with $r=k+3t-2$, which will give the stated bound by Theorem~\ref{thm:graphK2t}.  Indeed, assume for contradiction that $G$ contained such a $K_{2,r}$ on $\{x,y\} \cup \{u_1,\ldots,u_{r}\}$.  Let $S$ be the set of vertices $u_i$ of this $K_{2,r}$ for which $\{x,y,u_i\}\notin E(H)$, and by assumption there are at least $r-k=3t-2$ such vertices.  By definition of $S$, every vertex in $L_x,L_y$ has degree at least 1, so by Lemma~\ref{lem:min1} we can find a set of size at least $t$ that is dominated in $L_x,L_y$, giving a $K_{2,t}$ trace by Lemma~\ref{lem:simDom} which is a contradiction.
\end{proof}

\section{Proof of Lemma~\ref{lem:large}}\label{sec:large}

Our proof of Lemma~\ref{lem:large} involving hypergraphs with co-degrees at most $k$ will be an adaptation of a proof in~\cite{GMV} concerning linear hypergraphs. Throughout this section, unless stated otherwise we will assume to be working in $C_\delta$, which we recall is the set of edges in $H$ in which every pair has co-degree greater than $\delta$ in $H$.  For ease of notation we let $d(v) = d_{C_\delta}(v)$. We now begin the formal proof.

For $v\in V(C_\delta)$, define the 1 and 2-neighborhood of $v$ as \[N_1(v)=\{x\in V:\exists e\in E(C_\delta),\ v,x\in e\}.\]\[N_2(v)=\{x\in V(C_\delta)\sm (N_1(v)\cup \{v\}):\exists h\in E(C_\delta),\ x\in e,\ e\cap N_1(v)\ne \emptyset\}.\]
That is, $N_i(v)$ is the set of vertices that are distance $i$ from $v$. 

First observe that if $E$ is a set of edges containing some vertex $v$ and $V$ is the set of vertices $u\ne v$ with $u\in e$ for some $e\in E$, then 
\begin{equation}\label{lem:neigh}|V|\ge \f{2}{k}|E|,
\end{equation}
as each vertex in $V$ is contained in at most $k$ edges with $v$.

\begin{lem}\label{lem:expan}	For any $x\in V(C_\del)$ and $y\in N_1(x)$, the number of edges $e\in E(H)$ containing $y$ with $|e\cap N_1(x)|\ge 2$ is less than $k+\half k \cdot 50t$.
\end{lem}
\begin{proof}
Assume this was not the case for some $x,y$.  Note that at most $k$ of these edges contain $x$ since $\{x,y\}$ has co-degree at most $k$, so there exists a set of $\half k \cdot 50t$ of these edges $E$ which do not contain $v$.  Let $S=\bigcup_{e\in E} e\sm \{y\}$, and by~\eqref{lem:neigh} we have that $|S|\ge 50t$.  

In the language of the previous section, we define $L_x=L_x(H,S,y)$ and $L_y=L_y(H,S,x)$.  By definition of $C_\del$ and \eqref{eq:deg} these graphs have minimum degree at least $\del\ge 14$.  By Lemma~\ref{lem:minDel} we can find dominated sets $D_x,D_y$ of size at least $(1-\epsilon_\delta)|S| \geq .51|S|$, and thus $D=D_x\cap D_y$ is a set dominated in both $L_x,L_y$ of size at least $.02|S|\ge t$.  This implies that $H$ contains a $K_{2,t}$ trace by Lemma~\ref{lem:simDom}, a contradiction. 
\end{proof}

We point out that the above bound can be further optimized, however such improvements will not affect our asymptotic result.

From now on we fix some $v\in V(C_\delta)$.  For $u\in N_1(v)$, define \[E_u=\{e\in E(C_\del):e\cap N_1(v)=\{u\}\},\hspace{2em} V_u=\{w\in N_2(v):\exists e\in E_u,\ w\in e\}.\]

\begin{lem}\label{lem:Vi}
	\[\sum_{u \in N_1(v)} |V_u|\le ((1+4\epsilon)t-1)n.\]
\end{lem}
\begin{proof} Suppose for contradiction that $\sum |V_u| >((1+4\epsilon)t-1)n$. By the pigeonhole principle, there exists a vertex $x \notin N_1(v)$ and a set $S\sub N_1(v)$ of size at least $(1+4\ep)t$ such that $x\in V_u$ for all $u\in S$.  Define $L_v=L_v(H,S,x)$ and $L_x=L_x(H,S,v)$.  By assumption every $u\in S$ is contained in an edge $\{u,v,w_v\},\{u,x,w_x\}\in E(C_\del)$, so by \eqref{eq:deg} these graphs have minimum degree at least $\del$.  By Lemma~\ref{lem:minDel} we can find a set $D$ which is dominated in both of these graphs with size at least $(1-2\ep)(1+4\ep)t\ge t$ for $\del\ge 14$.  By Lemma~\ref{lem:simDom} we conclude that $H$ contains a $K_{2,t}$ trace, a contradiction.
\end{proof}

By Lemma~\ref{lem:expan} we have \[|E_u|\ge d(u)-k-\half k \cdot 50t +1\ge d(u)-26kt,\] and by \eqref{lem:neigh} we have $|V_u|\ge \f{2}{k}|E_u|$,  therefore
\begin{equation}\label{eq:Vu}
d(u)\le \f{k}{2}|V_u|+26kt.
\end{equation}
By Lemma~\ref{lem:Vi} and~\eqref{lem:neigh},
\[
\sum_{u \in N_1(v)} d(u) \leq \sum_{u \in N_1(v)} \l(\frac{k}{2} |V_u| + 26kt\r) \leq \frac{k}{2}\left(1 + 4\ep\right)tn + \frac{k}{2}d(v)\cdot 26kt.
\]

	Let $d=3e(C_\del)/n$ denote the average degree of $C_\del$. Then summing over the above inequality gives
\begin{equation}\label{eq:bound}
\sum_{v \in V(C_\del)} \sum_{u \in N_1(v)} d(u) \leq \frac{k}{2}\left(1 + 4\ep\right)tn^2 + 13k^2 t \cdot dn.
\end{equation}
	
	 	On the other hand, because $|N_1(u)|\ge \f{2}{k}d(u)$ by \eqref{lem:neigh}, and because $u\in N_1(v)$ if and only if $v\in N_1(u)$, we can reverse the sum to get
	 	
	 	\begin{equation}\label{eq:bound2}\sum_{u\in V(C_\del)}\sum_{v\in N_1(u)}d(u)\ge \sum_{u\in V(C_\del)} \f{2}{k}d(u)^2 \ge \f{2}{k}d^2n,
	 	\end{equation}
	with the last step following from the Cauchy-Schwarz inequality.  By combining \eqref{eq:bound} and \eqref{eq:bound2}, we find with $b:=\frac{13}{2}k^3t $ and $c:= \frac{k^2}{4}\left(1 +4\ep\right)t$ that 
	\[d^2-bd-cn\le 0\implies d\le \f{b+\sqrt{b^2+4cn}}{2}=\sqrt{c}n^{1/2}+O(1)\le \half k^{3/2},\]
	where this last step used $k\ge (1+4\ep)t$. Thus \[e(C_\del)=\f{dn}{3} \leq \frac{1}{6}k^{3/2}n^{3/2}+O(n),\]  giving the desired bound.

\section{Proof of Theorem~\ref{c4}}\label{sectionc4}
In this section we refine our methods and prove Theorem~\ref{c4} for forbidden $C_4$ traces. As many ideas are carried over from the proof of Theorem~\ref{k2t}, we omit some of the redundant details.  We note that the lower bound of Theorem~\ref{c4} follows from Theorem~\ref{thm:old}, so it remains to prove the upper bound.

Let $H$ be an $n$-vertex, $3$-uniform hypergraph with no $C_4$ trace. Let $A = H_1^-$, i.e., the edges with at least one pair of co-degree 1, and $B = H \setminus A$. 
Let $G_A$ be a graph on $[n]$ whose edge set is obtained by adding a pair of co-degree 1 from every edge of $A$. Then $G_A$ is $C_4$-free and we have \[|A| \leq ex(n, C_4) \le \half n^{3/2}+o(n^{3/2}).\]

It remains to show that $|B| \leq \frac{1}{3} n^{3/2} + o(n^{3/2})$. From now on we write $d_B(v),d_B(u,v)$ as $d(v),d(u,v)$. By Lemma~\ref{lem:cod}, $d(x,y) \leq 2$ for all $\{x,y\} \subset V(H)$. Similar to the proof of Lemma~\ref{lem:large}, for any vertex $v$ we let $N_1(v)$ and $N_2(v)$ denote the $1$- and $2$- neighborhoods of $v$ in $B$, respectively.

\begin{lem}\label{c4intersect}For any $x,y \in V(H)$, $|N_1(x)\cap N_1(y)|\leq 7.$\end{lem}
\begin{proof}Suppose there exists $x,y \in V(H)$ and some set $\{u_1, \ldots, u_8\} \subseteq N_1(x) \cap N_1(y)$. At most two $u_i$'s, say $u_7$ and $u_8$, are in edges of the form $\{x,y,u_i\}\in B$. Let $G$ be a graph on $[6]$ where $ij \in E(G)$ if and only if either $\{x,u_i, u_j\} \in B$ or $\{y, u_i, u_j\} \in B$. Because pairs in $B$ have co-degree at most 2, we have $\Delta(G) \leq 4$. In particular, there exists a non-adjacent pair, say $\{1,2\}$.  Let $e_{x,1}$ be any edge of $B$ containing $\{x,u_1\}$, and note that $y,u_2 \notin e_{x,1}$. Similarly define $e_{x,2},e_{y,1},e_{y_2}$. Then these four edges form a $C_4$ trace in $B$, which is a contradiction.
\end{proof} 

Now fix any vertex $v\in V(H)$. As before, define $E_u =\{e \in B: e \cap N_1(v) =\{u\}\}$ and $V_u= \{w \in N_2(v): \exists e \in E_u, w \in e\}$.  Since $V_u \subseteq N_1(u)$ for all $u$, we have the following corollaries. 

\begin{cor}\label{cor:7}Let $e = \{v, u, w\} \in B$ be any edge containing $v$. Then $|V_u \cap V_w| \leq 7$. 
\end{cor}

\begin{cor}\label{2intersect}
	For all $u\in N_1(v)$, \[|V_u|\ge d(u)-16\]
\end{cor}
\begin{proof}
	Note that $E_u$ consists of every edge containing $u$ except the at most 2 edges also containing $v$ and the edges $\{e\in B: u\in e, |e \cap N_1(v)|\geq 2\}$.  We claim that this latter set has cardinality at most 14.  Indeed, any such edge would contribute a vertex to $N_1(u) \cap N_1(v)$, of which there are at most 7 vertices by the previous lemma. Each such vertex can be contained in at most 2 edges with $u$ because $B$ has maximum co-degree at most 2.
	
	We conclude that $|E_u|\ge d(u)-16$, and because $B$ has maximum co-degree at most 2, $|V_u|\ge |E_u|$ by \eqref{lem:neigh}, giving the desired result.
\end{proof}

\begin{lem}\label{lem:Vi4}
	$\sum_{u \in N_1(v)} |V_u|\le n + 14d(v)$.
\end{lem}

\begin{proof}
Let $G_v:=\{xy: \{v,x,y\} \in B\}$ be the link graph of $v$. Because every pair has co-degree at most 2 in $B$, $\Delta(G_v) \leq 2$. 

\begin{claim}\label{deg1}Suppose that for some $u, w \in N_1(v)$, $V_u \cap V_w$ contains a vertex $x$. Then $B$ contains a $C_4$ trace on the vertices $v,u,x,w$ unless either $N_{G_v}(u) = \{w\}$ or $N_{G_v}(w) = \{u\}$. 
\end{claim}

\begin{proof}
Suppose that there exists edges $ua, wb \in E(G_v)$ such that $ua, wb \neq uw$. By the definition of $G_v$, $vua, vwb \in B$. Note that $a, b \neq x$, since $x \notin N_1(v)$. Let $e_u \in E_u$ and $e_w \in E_w$ be edges containing $x$. Then the edges $\{vua, e_u, e_w, vwb\}$ form a $C_4$ trace.
\end{proof}

We define the following sets $V_u' \subseteq V_u$ for $u \in N_1(v)$. If $d_{G_v}(u) = 2$, then $V_u' = V_u$. Otherwise if $N_{G_v}(u) = \{w\}$, set \[V_u' = V_u \setminus V_w.\] 
By Corollary~\ref{cor:7}, $|V_u'|\ge |V_u|-7$ for all $u$.  By Claim~\ref{deg1}, the $V_u'$ sets are pairwise disjoint from each other. Therefore $\sum_{u \in N_1(v)} |V_u'| \leq n$ and

\[\sum_{u \in N_1(v)} |V_u| \leq \sum_{u \in N_1(v)} (|V_u'| + 7) \leq n + 14 d(v),\]

where the last inequality comes from the fact that $|N_1(v)| \leq 2d(v)$.

\end{proof}

	By Corollary~\ref{2intersect} and Lemma~\ref{lem:Vi4}, we have \[\sum_{u\in N_1(v)} d(u)\le \sum_{u\in N_1(v)}(16 + |V_u|)\le 32d(v) +\sum_{u\in N_1(v)}|V_u| \leq n + 46d(v).\]
	
	Let $d=3e(H)/n$ denote the average degree of $H$.  We have \[d^2n \leq \sum_{u\in V(H)} d(u)^2 \leq \sum_{u\in V(H)}\sum_{v\in N_1(u)}d(u)= \sum_{v\in V(H)}\sum_{u\in N_1(v)} d(u)\le n^2 + 46dn.\]  Therefore $d \le \sqrt{n} + O(1)$, and hence $|B| =dn/3 \leq \frac{1}{3}n^{3/2} + O(n)$, as desired.  This completes the proof of the upper bound of Theorem~\ref{c4}.

\section{Concluding remarks}
It remains to determine the exact value of $ex(n, \Tr{C_4}{r})$, especially in the case where $r \geq 4$. The current best upper bound is that given by Theorem~\ref{thm:old}. In particular, we know
$\ex(n, \Tr{C_4}{r}) = \Theta(n^{3/2})$ for all $r$, but determining the limit (if it exists) $\lim_{n\to\infty} \ex(n, \Tr{C_4}{r})/n^{3/2}$ is likely difficult, though not as difficult as the more general $\ex(n,\Ber{C_4}{r})$ problem. For this problem, it is not even known if $\ex(n, \Ber{C_4}{r}) = \Theta (n^{3/2})$ for $r$ large.

{\bf Acknowledgement.} We thank Zolt\'an F\"uredi,  Abhishek Methuku, and Jacques Verstra\"ete for helpful comments.

\bibliographystyle{abbrv}
\bibliography{c4bib}

\begin{thebibliography}{10}

\bibitem{AS}
N.~Alon and C.~Shikhelman.
\newblock Many {$T$} copies in {$H$}-free graphs.
\newblock {\em J. Combin. Theory Ser. B}, 121:146--172, 2016.

\bibitem{alonspencer}
N.~Alon and J.~H. Spencer.
\newblock {\em The probabilistic method}.
\newblock John Wiley \& Sons, 2004.

\bibitem{BG}
B.~Bollob\'{a}s and E.~Gy\H{o}ri.
\newblock Pentagons vs. triangles.
\newblock {\em Discrete Math.}, 308(19):4332--4336, 2008.

\bibitem{BS}
J.~A. Bondy and M.~Simonovits.
\newblock Cycles of even length in graphs.
\newblock {\em J. Combinatorial Theory Ser. B}, 16:97--105, 1974.

\bibitem{brown}
W.~G. Brown.
\newblock On graphs that do not contain a {T}homsen graph.
\newblock {\em Canad. Math. Bull.}, 9:281--285, 1966.

\bibitem{ERS}
P.~Erd\H{o}s, A.~R\'{e}nyi, and V.~T. S\'{o}s.
\newblock On a problem of graph theory.
\newblock {\em Studia Sci. Math. Hungar.}, 1:215--235, 1966.

\bibitem{simonovits}
P.~Erd\H{o}s and M.~Simonovits.
\newblock A limit theorem in graph theory.
\newblock {\em Studia Sci. Math. Hungar.}, 1:51--57, 1966.

\bibitem{stone}
P.~Erd\"{o}s and A.~H. Stone.
\newblock On the structure of linear graphs.
\newblock {\em Bull. Amer. Math. Soc.}, 52:1087--1091, 1946.

\bibitem{ergem}
B.~Ergemlidze.
\newblock A note on maximum size of {B}erge-${C}_4$-free hypergraphs.
\newblock {\em arXiv 2001.01184}, 32, 2020.

\bibitem{EGM}
B.~Ergemlidze, E.~Győri, and A.~Methuku.
\newblock $3$-uniform hypergraphs and linear cycles.
\newblock 2016.
\newblock arXiv 1609.03934.

\bibitem{EGMTS}
B.~Ergemlidze, E.~Győri, A.~Methuku, C.~Tompkins, and N.~Salia.
\newblock On 3-uniform hypergraphs avoiding a cycle of length four.
\newblock {\em In preparation.}

\bibitem{furedi}
Z.~F\"{u}redi.
\newblock New asymptotics for bipartite {T}ur\'{a}n numbers.
\newblock {\em J. Combin. Theory Ser. A}, 75(1):141--144, 1996.

\bibitem{FL}
Z.~F\"uredi and R.~Luo.
\newblock Induced {T}urán problems and traces of hypergraphs.
\newblock 2020.
\newblock arXiv 2002.07350.

\bibitem{FO}
Z.~F{\"u}redi and L.~{\"O}zkahya.
\newblock On 3-uniform hypergraphs without a cycle of a given length.
\newblock {\em Discrete Applied Mathematics}, 216:582--588, 2017.

\bibitem{GMP}
D.~Gerbner, A.~Methuku, and C.~Palmer.
\newblock General lemmas for {B}erge-{T}ur\'{a}n hypergraph problems.
\newblock {\em European J. Combin.}, 86:103082, 15, 2020.

\bibitem{GMV}
D.~Gerbner, A.~Methuku, and M.~Vizer.
\newblock Asymptotics for the {T}ur\'{a}n number of {B}erge-{$K_{2, t}$}.
\newblock {\em J. Combin. Theory Ser. B}, 137:264--290, 2019.

\bibitem{GP}
D.~Gerbner and C.~Palmer.
\newblock Extremal results for {B}erge hypergraphs.
\newblock {\em SIAM J. Discrete Math.}, 31(4):2314--2327, 2017.

\bibitem{GMT}
D.~Gr\'{o}sz, A.~Methuku, and C.~Tompkins.
\newblock Uniformity thresholds for the asymptotic size of extremal
  {B}erge-{$F$}-free hypergraphs.
\newblock {\em European J. Combin.}, 88:103109, 2020.

\bibitem{GL2}
E.~Gy\H{o}ri and N.~Lemons.
\newblock 3-uniform hypergraphs avoiding a given odd cycle.
\newblock {\em Combinatorica}, 32(2):187--203, 2012.

\bibitem{GL}
E.~Gy\H{o}ri and N.~Lemons.
\newblock Hypergraphs with no cycle of a given length.
\newblock {\em Combin. Probab. Comput.}, 21(1-2):193--201, 2012.

\bibitem{KST}
T.~K\"{o}vari, V.~T. S\'{o}s, and P.~Tur\'{a}n.
\newblock On a problem of {K}. {Z}arankiewicz.
\newblock {\em Colloq. Math.}, 3:50--57, 1954.

\bibitem{M}
W.~Mantel.
\newblock Problem 28.
\newblock {\em Wiskundige Opgaven}, 10(60-61):320, 1907.

\bibitem{MZ}
D.~Mubayi and Y.~Zhao.
\newblock Forbidding complete hypergraphs as traces.
\newblock {\em Graphs Combin.}, 23(6):667--679, 2007.

\bibitem{SS}
A.~Sali and S.~Spiro.
\newblock Forbidden families of minimal quadratic and cubic configurations.
\newblock {\em Electron. J. Combin.}, 24(2):Paper No. 2.48, 28, 2017.

\bibitem{turan}
P.~Tur\'{a}n.
\newblock Eine {E}xtremalaufgabe aus der {G}raphentheorie.
\newblock {\em Mat. Fiz. Lapok}, 48:436--452, 1941.

\end{thebibliography}

\end{document}